\documentclass[reqno]{amsart}
\usepackage{amsmath}
\usepackage[dvips]{graphicx}
\usepackage{amsfonts}
\usepackage{amssymb}
\usepackage{latexsym}
\graphicspath{{Img/}}

\newtheorem{theorem}{Theorem}

\theoremstyle{plain}

\newtheorem{problem}{Problem}[]
\newtheorem{proposition}{Proposition}
\newtheorem{remark}{Remark}

\numberwithin{equation}{section}

\begin{document}
\title[A generalized Fermat-Torricelli tree on a surface]{A generalized Fermat-Torricelli tree that has acquired a subconscious on a surface}
\author{Anastasios Zachos}

\address{Greek Ministry of Education, Athens, Greece}
 \email{azachos@gmail.com}
\begin{abstract}
We study a generalized Fermat-Torricelli (S.FT) problem for
infinitesimal geodesic triangles on a $C^{2}$ complete surface $M$ with variable Gaussian curvature $a<K<b,$
for $a, b\in \mathbb{R},$ such that the intersection point (generalized Fermat-Torricelli point) of the three geodesics acquires a positive real number (subconscious).

The solution of the S.FT problem is a generalized Fermat-Torricelli tree with one node that has acquired a subconscious. This solution is based on a new variational method of the length of a geodesic arc with respect to arc length, which coincides with the first variational formula for geodesics on a surface with $K<0,$ or $0<K<c.$
The 'plasticity' solution of the inverse S.FT problem gives a connection of the absolute value of the Gaussian curvature $\|K(F)\|$ at the generalized Fermat-Torricelli point $F$ with the absolute value of the Aleksandrov curvature of the geodesic triangle by acquiring both of them the subconscious of the g.FT point.
\end{abstract}\maketitle

\section{introduction}

Let $\triangle A_{1}A_{2}A_{3}$ be a geodesic triangle on a $C^{2}$ complete surface $M.$
We denote by $w_{i}$ a positive real number (weight), which corresponds to each vertex $A_{i},$ by $l_{A_{i}}(F)$, the geodesic distance from the vertex $A_{i},$ to the point $F,$ for $i=1,2,3.$

The weighted Fermat-Torricelli problem on a $C^{2}$ complete surface $M$ states that:
\begin{problem}
Find a point $F\in M,$
such that:
\begin{displaymath}
 f(F)=w_{1}l_{A_{1}}(F)+w_{2}l_{A_{2}}(F)+w_{3}l_{A_{3}}(F)\to min.
\end{displaymath}

\end{problem}

The inverse weighted Fermat-Torricelli problem on $M$ states that:
\begin{problem}
Given a point $F$ which belongs to the interior of $\triangle A_{1}A_{2}A_{3}$ on $M,$
does there exist a unique set of positive weights $\{w_{1}, w_{2},w_{3}\},$ such
that
\begin{displaymath}
 w_{1}+w_{2}+w_{3}= c =const,
\end{displaymath}
for which $F$ minimizes
\begin{displaymath}
 f(F)=w_{1}l_{A_{1}}(F)+w_{2}l_{A_{2}}(F)+w_{3}l_{A_{3}}(F).
\end{displaymath}
\end{problem}

The solutions w.r to the weighted Fermat-Torricelli problem and an inverse weighted Fermat-Torricelli problem for a $C^{2}$ complete surface  with Gaussian curvature $0<K<c$ or $K<0,$  has been given in \cite{Zachos/Cots:10}, \cite{Cots/Zach:11}.

We mention the necessary and sufficient conditions to locate the weighted Fermat-Torricelli point at the interior of $\triangle A_{1}A_{2}A_{3}$ on a $C^{2}$ complete surface  with Gaussian curvature $0<K<c$ or $K<0:$

\begin{proposition}[Floating Case]\cite{Zachos/Cots:10},\cite{Cots/Zach:11}
If $P$, $Q$ $\in\{A_{1},A_{2},A_{3}\}$ and $\vec{U}_{PQ}$ is the unit tangent
vector of the geodesic arc $PQ$ at P and D is the domain of M
bounded by $\triangle A_{1}A_{2}A_{3},$
 then the following (I), (II), (III) conditions are equivalent:\\

(I) All the following inequalities are satisfied simultaneously:
\begin{equation}\label{cond120n}
\left\| w_{2}\vec{U}_{A_{1}A_{2}}+w_{3}\vec{U}_{A_{1}A_{3}}\right\|> w_{1},
\end{equation}

\begin{equation}\label{cond1202n}
\left\| w_{1}\vec{U}_{A_{2}A_{1}}+w_{3}\vec{U}_{A_{2}A_{3}}\right\|> w_{2},
\end{equation}

\begin{equation}\label{cond1203n}
\left\| w_{1}\vec{U}_{A_{3}A_{1}}+w_{2}\vec{U}_{A_{3}A_{2}}\right\|> w_{3},
\end{equation}
(II) The point $F$ is an interior point of the triangle
$\triangle A_{1}A_{2}A_{3}$ and does not belong to the geodesic arcs
$\gamma_{A_{1}A_{2}},$ $\gamma_{A_{2}A_{3}}$
and $\gamma_{A_{3}A_{1}},$\\

(III) $w_{1}\vec{U}_{FA_{1}}+w_{2}\vec{U}_{FA_{2}}+w_{3}\vec{U}_{FA_{3}}=\vec{0}.$\\

\end{proposition}

The solution of the weighted Fermat-Torricelli problem is a weighted tree (Weighted Fermat-Torricelli tree or weighted Steiner tree). The derivative of the weighted length of weighted Fermat-Torricelli trees and weighted Steiner trees on a connected complete Riemannian manifold is calculated in \cite{IvanovTuzhilin:01b} which is a generalization of the first variation formula for the length of geodesics w.r. to arc length (\cite{VToponogov:05}).
The weighted Fermat-Torricelli problem or weighted Steiner problem on a Riemannian manifold is a special case of a
one-dimensional variational problem in which branching extremals are introduced in \cite{IvanovTuzhilin:01}.

In this paper, we provide a new variational method to solve the weighted Fermat-Torricelli problem by assigning a positive number at the weighted Fermat-Torricelli point (g.FT that has acquired a subconscious)
for infinitesimal geodesic triangles on a $C^{2}$ complete surface $M$ with variable Gaussian curvature
$a<K<b,$ for $a, b\in \mathbb{R}.$

This variational method is based on the unified cosine law of Berg-Nikolaev given in \cite{BNik:07} for the K-plane (Sphere $S^{2}_{k}$, Hyperbolic plane $H^{2}_{k}$ and Euclidean Plane $\mathbb{R}^{2}$) and an assertion that the generalized Fermat-Torricelli point is located at three spherical regions or three hyperbolic regions or three plane regions or a combination of spherical, hyperbolic and plane regions with different constant curvatures.
Thus, we may obtain a generalized Fermat-Torricelli tree on a Torus or a surfaces of revolution in $\mathbb{R}^{3},$
having elliptic points ($K>0$,) hyperbolic points ($K<0$) and parabolic points $K=0$.

\section{The generalized Fermat-Torricelli(w.F-T)problem on a $C^{2}$ complete surface $M$ with $a<K<b$}

We denote by $\triangle ABC$ an infinitesimal geodesic triangle on a surface $M,$ by $w_{R}$ a positive real number (weight) which corresponds to each vertex $R,$ for $R\in \{A,B,C\}$ and by $w_{S}$ is a positive real number(weight)
which corresponds to an interior point $F$ of $\triangle ABC.$

The generalized Fermat-Torricelli problem with one node that has acquired unconscious (S.FT problem) states that:

Assume that we select weights $w_{A},$ $w_{B},$ $w_{C},$ such that the g.FT point is located at the interior of $\triangle ABC.$

\begin{problem}
Find the point $F\in M,$ that has acquired a subconscious $w_{S}$
such that:
\begin{equation}\label{minimum}
f(F)=w_{A}l_{A}(F)+w_{B}l_{B}(F)+w_{C}l_{C}(F)\to min.
\end{equation}

\end{problem}

We denote by $\varphi_{Q},$ the angle between the geodesic arcs
$\gamma_{RF}$ and $\gamma_{SF}$ for $Q,R,S\in\{A,B,C\}$
and $Q\ne R\ne S$.

\begin{theorem}\label{floatsol}
If the g.F-T point $F$ is an interior point of the infinitesimal geodesic triangle
$\triangle ABC$ (see figure 1), then each angle $\varphi_{Q},$
$Q\in\{A,B,C\}$ can be expressed as a function of $w_{A},$ $w_{B}$
and $w_{C}:$
\begin{equation}\label{eq:arr}
\cos\varphi_{Q}=\frac{w_{Q}^{2}-w_{R}^{2}-w_{S}^{2}}{2w_{R}w_{S}},
\end{equation}
for every $Q, R, S\in\{A,B,C\},$ $Q\ne R\ne S.$
\end{theorem}

\begin{proof}
Assume that $\triangle ABF,$ $\triangle BFC,$ $\triangle AFC$ belong to a spherical, hyperbolic or planar region
of constant Gaussian curvature $k_{3},$ $k_{1},$ $k_{2},$ for $k_{i}\in \mathbb{R},$ $i=1,2,3.$

We set $l_{A}(B^{\prime})\equiv l_{A}(F)+dl_{A}$

and $l_{B^{\prime}}(F)=dl_{B}.$

We denote by

\begin{displaymath}
\kappa_{i} = \left\{ \begin{array}{ll}
\sqrt{K_{i}} & \textrm{if $K_{i}>0$,}\\
i\sqrt{-K_{i}} & \textrm{if $K_{i}<0$.}\\
\end{array} \right.
\end{displaymath}

The unified cosine law for $\triangle AB^{\prime}F$ is given by:
\begin{equation} \label{eqvar1}
\cos(\kappa_{3} (l_{A}(F)+dl_{A}))=\cos(\kappa_{3} l_{A}(F))\cos(\kappa_{3}
dl_{B})+\sin(\kappa_{3} l_{A}(F))\sin(\kappa_{3}
dl_{B})\cos(\varphi_{C}),
\end{equation}

or

\begin{eqnarray} \label{eqvar1b}
\cos(\kappa_{3}l_{A}(F))\cos(\kappa_{3}dl_{A})-\sin(\kappa_{3}l_{A}(F))\sin(\kappa_{3}dl_{A})=\\\nonumber\cos(\kappa_{3} l_{A}(F))\cos(\kappa_{3}
dl_{B})+\sin(\kappa_{3} l_{A}(F))\sin(\kappa_{3}
dl_{B})\cos(\varphi_{C}),
\end{eqnarray}

By applying Taylor's formula, we obtain:

\begin{equation}\label{eqvar2}
\cos\kappa_{3} dl_{A}=1+o((k_{3}dl_{A})^{2}),
\end{equation}

\begin{equation}\label{eqvar3}
\sin\kappa_{3} dl_{A}=\kappa_{3}dl_{A}+o((k_{3}dl_{A})^{3}),
\end{equation}

\begin{equation}\label{eqvar4}
\cos\kappa_{3} dl_{B}=1+o((k_{3}dl_{B})^{2}),
\end{equation}
and
\begin{equation}\label{eqvar5}
\sin\kappa_{3} dl_{B}=\kappa_{3}dl_{B}+o((k_{3}dl_{B})^{3}).
\end{equation}

By replacing (\ref{eqvar2}), (\ref{eqvar3}),(\ref{eqvar4}),(\ref{eqvar5}) in (\ref{eqvar1b}) and neglecting second order terms, we derive that:

\begin{equation}\label{varphiC}
\frac{dl_{A}}{dl_{B}}=\cos(\pi-\varphi_{C}).
\end{equation}

The unified cosine law for $\triangle CB^{\prime}F$ is given by:
\begin{equation} \label{eqvar1c}
\cos(\kappa_{1} (l_{C}(F)+dl_{C}))=\cos(\kappa_{1} l_{C}(F))\cos(\kappa_{1}
dl_{B})+\sin(\kappa_{1} l_{C}(F))\sin(\kappa_{1}
dl_{B})\cos(\varphi_{A}),
\end{equation}

By applying Taylor's formula, we obtain:

\begin{equation}\label{eqvar2c}
\cos\kappa_{1} dl_{C}=1+o((k_{1}dl_{C})^{2}),
\end{equation}

\begin{equation}\label{eqvar3c}
\sin\kappa_{1} dl_{C}=\kappa_{1}dl_{C}+o((k_{1}dl_{C})^{3}),
\end{equation}

\begin{equation}\label{eqvar4c}
\cos\kappa_{1} dl_{B}=1+o((k_{1}dl_{B})^{2}),
\end{equation}
and
\begin{equation}\label{eqvar5c}
\sin\kappa_{1} dl_{B}=\kappa_{1}dl_{B}+o((k_{1}dl_{B})^{3}).
\end{equation}

Similarly, by replacing (\ref{eqvar2c}), (\ref{eqvar3c}),(\ref{eqvar4c}),(\ref{eqvar5c}) in (\ref{eqvar1c}) and neglecting second order terms, we derive that:

\begin{equation}\label{varphiA}
\frac{dl_{C}}{dl_{B}}=\cos(\pi-\varphi_{A}).
\end{equation}

By differentiating the objective function (\ref{minimum}) w.r. to a parameter $s,$ we get:

\begin{equation}\label{derobj1}
\frac{df}{ds}=w_{A}\frac{dl_{A}}{ds}+w_{B}\frac{dl_{B}}{ds}+w_{C}\frac{dl_{C}}{ds}
\end{equation}

By setting $s=-l_{B}$ and by replacing (\ref{varphiC}) and (\ref{varphiA}) in (\ref{derobj1}), we have:
\begin{equation}\label{equation1var}
w_{A}+w_{B}\cos(\varphi_{C}+w_{C}\cos(\varphi_{B})=0.
\end{equation}
Similarly, by working cyclically and setting the parametrization $s=-l_{C}$ and $s=-l_{A},$
we derive:

\begin{equation}\label{equation2var}
w_{A}\cos\varphi_{C}+w_{B}+w_{C}\cos\varphi_{A}=0,
\end{equation}

\begin{equation}\label{equation3var}
w_{A}\cos\varphi_{B}+w_{B}\cos\varphi_{A}+w_{C}=0,
\end{equation}

and

\[\varphi_{A}+\varphi_{B}+\varphi_{C}=2\pi.\]

The solution of (\ref{equation1var}),
(\ref{equation2var}) and (\ref{equation3var}) w.r. to $\cos\varphi_{Q}$  yields

(\ref{eq:arr}).

\end{proof}

Suppose that $w_{A},$ $w_{B},$ $w_{C}$ are variables and $\varphi_{A},$ $\varphi_{B},$ $\varphi_{C},$
are given.
The solution of (\ref{equation1var}),
(\ref{equation2var}) and (\ref{equation3var}) w.r. to $w_{A}, w_{B}, w_{C}$  yields

a positive answer w.r to the inverse weighted Fermat-Torricelli problem on $M:$

\begin{proposition}\label{propo5}
The solution of the inverse weighted Fermat-Torricelli problem on a surface $M$ is given by:
\begin{equation}\label{inverse111}
w_{Q}=\frac{Constant}{1+\frac{\sin{\varphi_{R}}}{\sin{\varphi_{Q}}}+\frac{\sin{\varphi_{S}}}{\sin{\varphi_{Q}}}},
\end{equation}
for $Q,R,S\in \{A,B,C\}.$
\end{proposition}

\begin{remark}
The solution of the inverse weighted Fermat-Torricelli problem on a $C^{2}$ complete surface with Gaussian curvature
$0<K<a$ or $K<0,$ has been derived in \cite{Zachos/Cots:10}, \cite{Cots/Zach:11}.
\end{remark}

The idea of assigning a residual weight (subconscious) at a weighted Fermat-Torricelli point (generalized Fermat-Torricelli point) is given in \cite{Zachos:20}, by assuming that a weighted Fermat-Torricelli tree is a
two way communication network and the weights $w_{A},$ $w_{B},$ $w_{C}$ are three small masses that may move through the branches of the weighted Fermat-Torricelli tree. By assuming mass flow continuity of this network, we obtain the generalized inverse weighted Fermat-Torricelli problem (inverse s.FT problem).

The inverse s.F.T problem is the inverse weighted Fermat-Torricelli problem, such that the weighted Fermat-Torricelli point has acquired a subconscious $w_{S}.$

We denote by $w_{R}$ a mass flow which is transferred from $R$
to $F$ for $R\in \{A,B\},$ by $w_{S}$ a residual weight which
remains at $F$ and by $w_{C}$ a mass flow which is transferred
from $F$ to $C,$ by $\tilde{w_{R}}$ a mass flow which is transferred from
$F$ to $R,$ $R\in \{A,B\},$ and by $\tilde{w_{S}}$ a residual
weight which remains at $F$ and by $\tilde{w_{C}}$ a mass flow
which is transferred from $C$ to $F.$

The following equations are derived by this mass flow along the infinitesimal geodesic arcs $AF,$ $BF,$ $CF:$

\begin{equation}\label{weight1outflow}
w_{A}+w_{B}=w_{C}+w_{S}
\end{equation}

and

\begin{equation}\label{weight2inflow}
\tilde{w_{A}}+\tilde{w_{B}}+\tilde{w_{S}}=\tilde{w_{C}}.
\end{equation}

By taking into account (\ref{weight1outflow}) and (\ref{weight2inflow}) and by
setting $\bar{w_{S}}=w_{S}-\tilde{w_{S}},$ we get:

\begin{equation}\label{weight12inoutflow}
\bar{w_{A}}+\bar{w_{B}}=\bar{w_{C}}+\bar{w_{S}}
\end{equation}

such that:

\begin{equation}\label{weight12inflowsum}
\bar{w_{A}}+\bar{w_{B}}+\bar{w_{C}}=c>0,
\end{equation}

\begin{problem}\label{mixinv5triangle}
Given a point $F$ which belongs to the interior of the inifinitesimal geodesic triangle $\triangle
ABC$ on $M$, does there exist a unique
set of positive weights $\bar{w_{R}},$ such that
\begin{equation}\label{isoptriangle}
 \bar{w_{A}}+\bar{w_{B}}+\bar{w_{C}} = c =const,
\end{equation}
for which $F$ minimizes
\begin{displaymath}
 f(F)=w_{A}l_{A}(F)+w_{B}l_{B}(F)+w_{C}l_{C}(F),
\end{displaymath}
\begin{displaymath}
 f(F)=\tilde{w}_{A}l_{A}(F)+\tilde{w}_{B}l_{B}(F)+\tilde{w}_{C}l_{C}(F),
\end{displaymath}
\begin{displaymath}
 f(F)=\bar{w}_{A}l_{A}(F)+\bar{w}_{B}l_{B}(F)+\bar{w}_{C}l_{C}(F),
\end{displaymath}

\begin{equation}\label{imp1mixtr}
w_{R}+\tilde{w_{R}}=\bar{w_{R}}
\end{equation}
under the condition for the weights:

\begin{equation}\label{cond3mixtr}
\bar{w_{i}}+\bar{w_{j}}=\bar{w_{S}}+\bar{w_{k}}
\end{equation}
for $i,j,k\in {A,B,C}$ and $i\ne j\ne k.$
\end{problem}

\begin{theorem}\label{propomix4triangle}
Given the g.FT point $F$ to be an
interior point of the triangle $\triangle ABC$ with
the vertices lie on three geodesic arcs that meet at $F$ and
from the two given values of $\varphi_{B},$ $\varphi_{C},$ the
positive real weights $\bar{w_{R}}$ given by the formulas

\begin{equation}\label{inversemix42tr}
\bar{w_{A}}=-\left(\frac{\sin(\varphi_{B}+\varphi_{C})}{\sin\varphi_{C}}\right)\frac{c-\bar{w_{S}}}{2},
\end{equation}
\begin{equation}\label{inversemix43tr}
\bar{w_{B}}=\left(\frac{\sin\varphi_{B}}{\sin\varphi_{C}}\right)\frac{c-\bar{w_{S}}}{2},
\end{equation}
and
\begin{equation}\label{inversemix41tr}
\bar{w_{C}}=\frac{c-\bar{w_{S}}}{2}
\end{equation}
give a negative answer w.r. to the inverse s.FT problem on $M.$
\end{theorem}

\begin{remark}
Theorem~2 is proved in \cite{Zachos:20} for the case of $\mathbb{R}^{2}.$
\end{remark}

We conclude with an evolutionary scheme of infinitesimal geodesic triangles, which connects the subconscious
of a weighted Fermat-Torricelli tree with the Aleksandrov curvature of a geodesic triangle (\cite{Alexandrov:96},\cite{BNik:07})).

Phase~1 At time zero, we assume that a point $F$ in $\mathbb{R}^{3}$ tends to split in three directions.
It acquires a subconscious which equals with the absolute value of the Gaussian curvature $\|K(F)\|$ and predermines the surface with Gaussian curvature $K,$ on which these three geodesic arcs will move (Weighted Fermat-Torricelli tree).

Phase~2 After time $t,$ the subconscious quantity is increased and the value of the Aleksandrov curvature of the infinitesimal geodesic triangle $T$ is reached:
\[\bar{w_{S}}=\|K(T)\|=\|\angle A+\angle B+\angle C-\pi\|.\]

The following equations determine the values of $\bar{w}_{A},$ $\bar{w}_{B}$ and $\bar{w}_{C}:$

\[\bar{w_{A}}=-\left(\frac{\sin(\varphi_{B}+\varphi_{C})}{\sin\varphi_{C}}\right)\frac{1-\bar{w_{S}}}{2},\]

\[\bar{w_{B}}=\left(\frac{\sin\varphi_{B}}{\sin\varphi_{C}}\right)\frac{1-\bar{w_{S}}}{2},\]

and

\[\bar{w_{C}}=\frac{1-\bar{w_{S}}}{2}.\]

Phase~2 gives a plasticity solution of the inverse s.FT problem that has acquired a subconscious $\bar{w_{S}}=\|K(\triangle ABC)\|.$



\begin{thebibliography}{99}
\bibitem{Alexandrov:96} A. D. Alexandrov, \emph{A.D. Alexandrov Selected Works Part I Selected Scientific Papers}, Gordon and Breach Publishers, Amsterdam, 1996.
\bibitem{BNik:07}I.D. Berg and I.G. Nikolaev, \emph{ On an extremal property
of quadrilaterals in an Aleksandrov space of curvature $\leq K$.
The interaction of analysis and geometry}, Contemp. Math. \textbf{
424},(2007),  1-15.



\bibitem{IvanovTuzhilin:01}A.O. Ivanov and A.A. Tuzhilin,
\emph{Branching solutions to one-dimensional variational problems.}
Singapore: World Scientific 2001.

\bibitem{IvanovTuzhilin:01b}A.O. Ivanov and A.A. Tuzhilin,
\emph{Differential calculus on the space of Steiner minimal trees in Riemannian manifolds},
Sb. Math. \textbf{192} (2001), no. 6, 823-841, translation from Mat. Sb. \textbf{192} (2001), no. 6, 31-50.



\bibitem{VToponogov:05} V.A. Toponogov, \emph{Differential geometry of curves and
surfaces}, Birkhauser, 2005.

\bibitem{Zachos/Cots:10} A. Zachos and A. Cotsiolis, \emph{The weighted Fermat-Torricelli problem on a surface and an "inverse" problem},  \emph{J. Math. Anal. Appl.}, \textbf{373}, no. 1 (2011)  44--58.
\bibitem{Cots/Zach:11} A. Cotsiolis and A. Zachos, \emph{Corrigendum to "The weighted Fermat-Torricelli problem on a surface and an "inverse"
problem"}, \emph{J. Math. Anal. Appl.}, \textbf{376}, no. 2 (2011)
760.

\bibitem{Zachos:14} A. Zachos, \emph{A plasticity principle of convex quadrilaterals on a convex surface of bounded specific
curvature}, Acta. Appl. Math., \textbf{129}, no. 1 (2014), 81--134.

\bibitem{Zachos:20} A.N. Zachos, \emph{The Plasticity of some Mass Transportation Networks in the Three Dimensional Euclidean Space}, J. Convex Anal. \textbf{27}, no. 3 (2020), To appear.


\end{thebibliography}
\end{document}